\documentclass[a4paper,12pt]{amsart}
\usepackage[utf8]{inputenc}
\usepackage{amsmath,amsthm,amsfonts,amssymb}
\usepackage{color}
\usepackage{graphicx}
\usepackage{hyperref}
\hypersetup{
    unicode=false,          
    pdftoolbar=true,        
    pdfmenubar=true,        
    pdffitwindow=false,     
    pdfstartview={FitH},    
    pdftitle={},    
    pdfauthor={Soumendu Sundar Mukherjee},     
    pdfsubject={Mathematics},   
    pdfcreator={Creator},   
    pdfproducer={Producer}, 
    pdfkeywords={keyword1} {key2} {key3}, 
    pdfnewwindow=true,      
    colorlinks=true,       
    linktoc=page,
    linkcolor=blue,          
    citecolor=blue,        
    filecolor=magenta,      
    urlcolor=cyan,           
    linkbordercolor={1 1 1},
    citebordercolor={0 1 0},
    urlbordercolor={0 1 1}
}

\definecolor{purple}{rgb}{0.9,0,0.8}

\definecolor{gray}{rgb}{0.5,0.5,0.5}

\newtheorem{prop}{Proposition}[section]
\theoremstyle{plain}

\theoremstyle{plain}
\newtheorem{theorem}{Theorem}[section]
\theoremstyle{plain}
\newtheorem{lemma}{Lemma}[section]
\theoremstyle{plain}
\newtheorem{cor}{Corollary}[section]
\theoremstyle{plain}
\newtheorem{exm}{Example}[section]
\theoremstyle{remark}
\newtheorem{remark}{Remark}[section]
\theoremstyle{remark}

\newcommand{\E}{\mathbb{E}}
\newcommand{\R}{\mathbb{R}}

\newcommand{\N}{\mathbb{N}}
\newcommand{\Z}{\mathbb{Z}}

\newcommand{\I}{\mathbb{I}}
\newcommand{\tr}{\text{tr}}

\numberwithin{equation}{section}

\title[LSD of a class of Hankel type random matrices]{Limiting Spectral Distribution of a class of Hankel type random matrices}

\author[A. Basak]{Anirban Basak}
\address{Department of Statistics, Stanford University, Stanford, California 94305}
\email{anirbanb@stanford.edu}

\author[A. Bose]{Arup Bose}
\address{Statistics and Mathematics Unit, Indian Statistical Institute, Kolkata 700108}
\email{bosearu@gmail.com}
\author[S. S. Mukherjee]{Soumendu Sundar Mukherjee}
\address{Indian Statistical Institute, Kolkata 700108}
\email{soumendu041@gmail.com}
\keywords{Hankel matrix, Reverse Circulant matrix, symmetrized Rayleigh distribution, method of moments}
\subjclass[2010]{Primary 15B52, 60B20; secondary 60B10, 60F99, 60B99.}
\date{July 30, 2014}

\begin{document}

\begin{abstract}  We consider an indexed class of real symmetric random matrices which generalize the  symmetric Hankel and Reverse Circulant matrices. We show that the limiting spectral distributions of these matrices exist almost surely and the limit is continuous in the index. We also study other properties of the limit.
\end{abstract}
\maketitle

\section{Introduction}
For an $n\times n$ real symmetric random
matrix $A_n$, let $\lambda_1, \ldots , \lambda_n \in \R $ 
 be its eigenvalues.
The empirical spectral measure $\mu_n$ of $A_n$ is the
random
measure 
\begin{equation}
\mu_n:= \frac{1}{n} \sum_{i=1}^n \delta_{ \lambda_i},
\end{equation}
where $\delta_{x}$ is the Dirac delta measure at $x$. The
corresponding random
probability distribution function  is known as the \emph{Empirical
Spectral Distribution} (ESD) and  is denoted by
$F^{A_n}$. The sequence $\{F^{A_n}\}$ is said to converge (weakly) almost surely to a non-random 
distribution function $F$, if	 outside a null set, as $n \to \infty$, $F^{A_n}(x) \rightarrow F(x)$ for all
continuity points of $F$, and then $F$  is known as the limiting spectral distribution (LSD).

There has been a lot of recent work on obtaining the LSD  of
large dimensional patterned random matrices, including but not limited to  \cite{bose2008another, bryc2006spectral, basu2012joint, bose2011convergence, bose2014bulk, jackson2012distribution, kargin2009spectrum, massey2007distribution, meckes2009some}. These matrices may be defined as follows.
Let $\{L_n\}$ be a sequence of \textit{link} functions such that
\begin{equation}\label{eq:link}
L_n: \{1, 2, \ldots n\}^2 \to \mathbb Z, \  n \geq 1, 
\end{equation}
where $\Z$ denotes the set of all integers. For notational simplicity we write $L$ for $L_n$ and write $\Z_+^2:=\Z_+ \times \Z_+$,
 where $\Z_+$ denotes the set of all non-negative integers, 
 as the common domain of $\{L_n\}$. Furthermore, let  $\{ x_i; i \geq 0\}$ be an  \textit{input sequence}  of random variables.
 Then matrices  of the form
\begin{equation}\label{eq:patternmatrix}
A_n=n^{-1/2}((x_{L(i,j)}))
\end{equation}
are called \textit{patterned matrices}. If
$L(i,j)=L(j,i)$ for all $i,j$, then the matrix is symmetric. 

Two symmetric patterned matrices that have received particular attention in recent times are the Hankel matrix and the Reverse Circulant matrix. For example, see \cite{basak2010balanced, basak2011limiting, bryc2006spectral, bose2009limiting, bose2002limiting, bose2012limiting, li2011limit, liu2011limit}, and the references in \cite{bose2010patterned}. The Hankel and the Reverse Circulant matrices are given 
by 
\[
H_{n}  = \frac{1}{\sqrt{n}}\left[ \begin{array} {cccccc}
            x_{2} & x_{3} & x_{4} & \ldots & x_{n} & x_{n+1} \\
            x_{3} & x_{4} & x_{5} & \ldots & x_{n+1} & x_{n+2} \\
            x_{4} & x_{5} & x_{6} & \ldots & x_{n+2} & x_{n+3} \\
                  &       &       & \vdots &       &   \\
            x_{n+1} & x_{n+2} & x_{n+3} & \ldots & x_{2n-1} & x_{2n}
            \end{array} \right]
\]
and
\[
  RC_{n}  =\frac{1}{\sqrt{n}}\left[ \begin{array} {cccccc}
            x_{2} & x_{3} & x_{4} & \ldots & x_{0} & x_{1} \\
            x_{3} & x_{4} & x_{5} & \ldots & x_{1} & x_{2} \\
            x_{4} & x_{5} & x_{6} & \ldots & x_{2} & x_{3} \\
                  &       &       & \vdots &       &   \\
            x_{1} & x_{2} & x_{3} & \ldots & x_{n-1} & x_{0}
            \end{array} \right]
\]
respectively obtained with the link functions  $L_H(i,j) =i+j$  and 
$L_{RC}(i,j) =(i+j )\mod n$. These matrices are similar but have different LSD.
While the LSD of the Reverse Circulant is known explicitly (it is the symmetrized Rayleigh distribution with density $f(x)=|x| \exp{(-x^2)}, -\infty < x < \infty$), very little is known about the LSD of the Hankel. It is known that it  is not unimodal \cite{bryc2006spectral}, and simulations show that it is actually bimodal.  However nothing else is known about the limit.

Due to the similarity between the matrices, it is natural to ask whether one can bring them under a common class of matrices and obtain a relation between the two LSDs  or, if one can move seamlessly from one to the other. To this end, consider the class of link functions 
\begin{equation}\label{link}
L_{\theta}(i,j)=i+j \, (\text{mod}\lfloor n/\theta \rfloor),
\end{equation}
where $\theta \in (0,\infty)$ and $\lfloor a \rfloor$ denotes the
greatest integer less than or equal to $a$. Let $A_n^\theta$ be the matrix  corresponding to the link function $L_\theta$. 
The Hankel and Reverse Circulant matrices  are obtained respectively when $\theta \leqslant 1/2$  and $\theta=1$. 
Our aim is to explore the existence of the LSD for patterned matrices with this class of link functions for all $\theta$ and investigate the properties of these LSDs, specially in comparison to the LSDs of Hankel and Reverse Circulant. 

We show that for all $\theta\in(0,\infty)$ the LSD exists, is symmetric about zero, and is universal when the input sequence is i.i.d.~with mean zero and variance one (see Theorem \ref{thm:lsd}(i)). Moreover, the moments of the LSD are dominated by some Gaussian moments and are always at least as large as the moments of the corresponding Hankel LSD. Additionally, we show that if $\theta \leqslant 1$, the $2k$-th moment is dominated by $k!$, which is the $2k$-th moment of the Reverse Circulant. 
 
 One expects that the discontinuity of $\lfloor \cdot \rfloor$ would be washed away in the limit, and the LSDs must be continuous, when viewed as a function of $\theta$. We confirm this intuition in Theorem \ref{thm:lsd}(ii). Further, when $\theta$ is an integer, we explicitly identify the  LSD in Theorem \ref{thm:theta_integer}. 
 
For general $\theta \in (0,\infty)$, we can also derive some information about the so called word limits $p_\theta(w)$. For the patterned matrices,  Hankel, Reverse Circulant, Toeplitz, Wigner etc.,
 it is known that for every Catalan word $w$ (see Section \ref{sec:wordlimits} for definition), $p(w)=1$. We prove in this article that for $\theta\leq 1$ one still has $p_\theta(w)=1$ for each Catalan word $w$. However, for $\theta>1$ they are not even equal and each $p_\theta(w) > 1$. Further, as $\theta \rightarrow \infty$, for every Catalan word $w$ of length $2k$,  $p_{\theta}(w)\sim \theta^{k-1}$. We also provide a recursive relation for computing $p_\theta(w)$, when $w$ is a Catalan word (see Section \ref{sec:wordlimits} for more details). {Extending the ideas in \cite{bryc2006spectral} one can prove that the limits are not unimodal. However we do not pursue that direction in this paper. Simulations also show that the LSD, when the mass at zero is removed, is bimodal for all $\theta$. 

Here is the outline of the paper:  in Section \ref{sec:prelim_main_results} we introduce the necessary terminology, and state the main results. Section \ref{sec:thelsd} is devoted mainly to the proof of Theorems \ref{thm:repr} and \ref{thm:lsd}. In Section \ref{sec:wordlimits} we prove Theorem \ref{thm:gen_hankel_cat}, which provides a recursive relation for computing $p_\theta(w)$, when $w$ is a Catalan word, followed by the statements and proofs of several corollaries which identify some behavior of the Catalan word limits in different regimes of $\theta$.  This section also includes a proof of Theorem \ref{thm:theta_integer}. 

\section{Preliminaries and Main Results}\label{sec:prelim_main_results}
We shall use the method of moments to establish the existence of the LSD. For any matrix $A$, let $\beta_h(A)$ denote the $h$-th moment of the ESD of $A$. 
We quote the following lemma which is easy to prove \cite[see for example,][]{bose2008another}. 
\begin{lemma}\label{mainlemma}
Fix $\theta \in (0,\infty)$. Let $\{A_n^\theta\}$ be the sequence of real symmetric random matrices formed from an input sequence $\{x_i: i \ge 0\}$, via the link function $L_\theta$. Suppose there exists a sequence $\{\beta_h^{(\theta)}\}$ such that 

\noindent 
(i) for every $h\geq 1$, $\E(\beta_h(A_n^\theta))\rightarrow \beta_h^{(\theta)}$, 

\noindent (ii) for every $h\geq 1$  $\sum_{n=1}^{\infty}\E[\beta_h(A_n^\theta)-\E(\beta_h(A_n^\theta))]^4<\infty$ and 

\noindent 
(iii) the sequence $\{\beta_h^\theta\}$ satisfies Carleman's condition,  i.e. $\sum_{h=1}^\infty(\beta_{2h}^{(\theta)})^{-1/2h}=\infty$.

\noindent
Then  the LSD of $F^{A_n^\theta}$ exists and equals $F^\theta$ with moments $\{\beta_h^{(\theta)}\}$.
\end{lemma}

\noindent
Our goal will be to establish all the hypotheses of Lemma \ref{mainlemma}, and then the existence of the LSD will follow. To claim the existence of the LSD for all $\theta$, we shall make use of the general notation and theory developed in \cite{bose2008another} for patterned matrices.  First observe that $L_\theta$  satisfies the so called {\bf Property B} -- the total number of times any particular variable appears in any row is uniformly bounded. Moreover, the total number of different variables in the matrix and the total number of times any variable appears in the matrix are both of the order $n$. This implies that the general theory developed in \cite{bose2008another} applies to this class of link functions.\\\\
\noindent
The LSD of patterned matrices is usually studied under one of the following assumptions:

\vskip10pt

\noindent 
{\bf Assumption 1.}  $\{x_i\}$  are independent and uniformly bounded with mean $0$, and variance $1$.\vskip10pt

\noindent 
{\bf Assumption 2.}   $\{x_i\}$  are i.i.d.~with mean $0$ and variance $1$.\vskip10pt

\noindent 
{\bf Assumption 3.}   $\{x_i\}$  are independent with mean $0$ and variance $1$, and with uniformly bounded moments of all orders.

\vskip10pt

\noindent
If the LSD exists under Assumption 1, then using a truncation argument it can be shown that the same LSD continues to hold under Assumptions 2 or 3. See for instance Theorems 1 and 2 of \cite{bose2008another}. Thus all our  proofs will be presented only under Assumption 1.  
Traditionally LSD results are stated under Assumption 1, and Assumption 3 is appropriate while studying the joint convergence of more than one sequence of matrices. 
\vskip10pt

\noindent
To use Lemma \ref{mainlemma}, one needs to guarantee the limits of $\{\E[\beta_h(A_n^\theta)]\}_{n=1}^\infty$ for all $h \ge 1$. 
We first provide a better  representation of $\E[\beta_h(A_n^\theta)]$  via the \emph{Moment-Trace Formula} using 
the terminology introduced in \cite[Section 3]{bose2008another}. 
We call a function $$\pi : \{0, 1, \ldots, h\} \rightarrow \{1,2,\ldots,n\}$$ with $\pi(0) = \pi(h)$ a \emph{circuit} 
of
length $h$. For notational simplicity suppressing the dependence of a circuit on $h$ and $n$ we note that
\begin{equation}\label{eq:momtr}
\E[\beta_h(A_n^\theta)]= \frac{1}{n}\E\tr[(A_n^\theta)^h]=\frac{1}{n}\sum_{\pi \text{ circuit of length $h$}}\E[a_\pi^\theta],
\end{equation}
where 
\[A_n^\theta=((a_{L_\theta(i,j)})), \ 
a_{\pi}^\theta := a_{L_\theta(\pi(0),\pi(1))}a_{L_\theta(\pi(1),\pi(2))}\ldots a_{L_\theta(\pi(h-1),\pi(h))}.
\]

To simplify the expression further we identify the pair $(i,j)$ for which their $L_\theta$-values match, i.e. $L_\theta(\pi(i-1),\pi(i))=L_\theta(\pi(j-1),\pi(j))$, with $i<j$.
From \cite[Lemma 1]{bose2008another} and the fact that all the entries in the random matrix $A_n^\theta$ have zero mean it follows that circuits where there are only $\theta$-pair-matches\footnote{A circuit is  $\theta$-pair-matched if all the $L_\theta$-values are repeated exactly twice.} are relevant when computing limits of moments. So it is enough to consider the summation over all $\theta$-pair-matched circuits of length $h$ in (\ref{eq:momtr}).

Now for every $\theta \in (0,\infty)$, we define a $\theta$-equivalence relation on the set of all circuits of length $h$, such that 
 $\E a_\pi^\theta$ are same on each equivalence class: two circuits $\pi_1$ and $\pi_2$ are $\theta$-equivalent if and only if their $L_\theta$-values respectively match at the
same locations, i.e. if for all $i, j$,
\[
L_\theta(\pi_1(i-1),\pi_1(i))=L_\theta(\pi_1(j-1),\pi_1(j)) \Leftrightarrow L_\theta(\pi_2(i-1),\pi_2(i))=L_\theta(\pi_2(j-1),\pi_2(j)).
\]

Thus any $\theta$-equivalence class can be indexed by a partition of $\{1,2,\ldots, h\}$. 
We identify these partitions with \emph{words} of length $l(w) := h$ of letters, where every member of a partition block is denoted by a single letter, and the first occurrence
of each letter is in alphabetical order. For example if $h = 4$ then the partition
$\{\{1,3\},\{2,4\}\}$ is represented by the word $abab$. For a $\theta \in (0,\infty)$, this identifies all circuits $\pi$ for
which $L_\theta(\pi(0),\pi(1)) = L_\theta(\pi(2), \pi(3))$ and $L_\theta(\pi(1),\pi(2)) = L_\theta(\pi(3),\pi(4))$.
Therefore, denoting $w[i]$ to be the $i^{th}$ entry of $w$, the $\theta$-equivalence class corresponding to $w$ can be written as 
\[
\Pi_\theta(w) := \{\pi \mid w[i]=w[j]\Leftrightarrow L_\theta(\pi(i-1),\pi(i))=L_\theta(\pi(j-1),\pi(j))\}.
\]
The number of partition blocks corresponding to $w$ equals the
number of distinct letters in $w$, say $|w|$ and for any $\pi\in \Pi_\theta(w)$,
\[
|w| = \#\{L_\theta(\pi(i-1),\pi(i))\mid 1\leqslant i \leqslant h \},
\]
where $\#\mathcal{A}$ denotes the cardinality of the set $\mathcal{A}$. 
Note that for any
fixed $h$ even as $n\rightarrow \infty$, the number of words (equivalence classes) remains finite but
the number of circuits in any given $\Pi_\theta(w)$ may grow indefinitely. Henceforth we shall denote the set of all words of length $h$ by $\mathcal{W}_{h}$. 
Notions of matches carry over to words. 
A word is \emph{pair-matched} if every letter 
appears exactly twice in that word. The set of all pair-matched words of length $2k$
is  denoted by $\mathcal{W}_{2k}^{(p)}$.
For technical reasons it is often easier to deal with a class larger than $\Pi_\theta(w)$:
\[
\Pi^*_\theta(w) = \{\pi \mid w[i]=w[j]\Rightarrow L_\theta(\pi(i-1),\pi(i))=L_\theta(\pi(j-1),\pi(j))\}.
\]
From \cite[Lemma 1(b)(ii)]{bose2008another} it follows that, for any $\theta \in (0,\infty)$, every $k \ge 1$, and every $w\in \mathcal{W}_{2k}^{(p)}$, $\lim_{n \rightarrow \infty}n^{-(k+1)} |\#(\Pi_\theta(w) \setminus\Pi^*_\theta(w))|=0$. Furthermore noting that by varying $w$, we obtain all the equivalence classes, and $\E(a_\pi^\theta
)=n^{-k}$ for any $\theta$-pair matched circuit of length $2k$, we deduce that
\begin{equation}
\lim_{n \rightarrow \infty} \E[\beta_{2k}(A_n^\theta)] = \sum_{w \in \mathcal{W}_{2k}^{(p)}} \lim_{n \rightarrow \infty} n^{-(k+1)}\# \Pi_\theta^*(w),
\end{equation}
provided the limits in the right  side exist.  That they  exist  is the contention of Theorem \ref{thm:repr}. 

To state this theorem, we need the following
 notions: any $i$ (or $\pi(i)$ by abuse of notation) is  a \emph{vertex}. It is \emph{generating} if either $i = 0$
or $w[i]$ is the first occurrence of a letter. Otherwise, it is non-generating. For
example, if $w = abbcab$ then $\pi(0), \pi(1), \pi(2), \pi(4)$ are generating and $\pi(3), \pi(5), \pi(6)$
are non-generating. The set of generating vertices (indices) is denoted by $S(w)$.
\begin{theorem}\label{thm:repr} Fix $\theta \in (0,\infty)$ and $k \ge 1$. Then
 for each $w \in \mathcal{W}_{2k}^{(p)}$, we have 
\begin{align}\label{gen_hankel_word_limit}
\nonumber p_\theta(w)&=\lim \frac{1}{n^{1+k}}\Pi^*_\theta(w)\\
\nonumber &= \sum_{\gamma\in \mathcal{A}_{\theta}^k}\theta^{-k-1} \underbrace{\int_{0}^{\theta}\cdots \int_{0}^{\theta}}_{|w|+1 \text{ fold}} \I\, (L_i^H(\nu_{S(w)})+a_i^{(\gamma)} \in (0,\theta), \text{ for $i \notin S(w) \cup \{2k\}$})\\
& \hskip180pt \times \I(\nu_0=L_{2k}^H(\nu_{S(w)})+a_{2k}^{(\gamma)})\, \text{d}\nu_{S(w)},
\end{align}
where $\mathcal{A}_{\theta}=\{0,\pm 1,\ldots \pm \lfloor 2\theta \rfloor \}$, $\{a_i^{(\gamma)}\}_{i \notin S(w)}$ are some constants depending on $\gamma$ (and possibly also on the word $w$), and $\{L_i^H(\cdot)\}_{i \notin S(w)}$ are some linear functions of the variables $\nu_{S(w)}=\{\nu_i: i\in S(w)\}$.
\end{theorem}
Now building on Theorem \ref{thm:repr}, and using Lemma \ref{mainlemma} we obtain the following.
\begin{theorem}\label{thm:lsd}
Fix $\theta \in (0,\infty)$, and let the input sequence satisfy Assumptions $1$, $2$ or $3$. Then the following hold.

\noindent
(i) The LSD (say   $F_\theta$) of $A_n^\theta$   exists. Moreover denoting $\beta_h^{(\theta)}= \int x^h dF_\theta(x)$, we have that $$\beta_{2k-1}^{(\theta)}=0, \ \ \text{and }\ \ \beta_{2k}^{(\theta)}= \sum_{w \in \mathcal{W}_{2k}^{(p)}} p_\theta(w), \text{ for all }k \ge 1,$$ where $p_\theta(w)$ is given by (\ref{gen_hankel_word_limit}).\vskip5pt

\noindent
(ii) If $\theta_m \rightarrow \theta \in (0, \infty)$, then  $F_{\theta_m}$ converge weakly to $F_\theta$.
\end{theorem}
Since the input sequence satisfies Property B for all $\theta \in (0, \infty)$, using Lemma 1(b)(i) of \cite{bose2008another}, one concludes that $\beta_{2k-1}^{(\theta)}=0$ for all $k \ge 1$. Therefore this, together with Theorem \ref{thm:repr}, establishes   condition (i) of Lemma \ref{mainlemma}. Furthermore, from Lemma 2 and Theorem 3 of \cite{bose2008another}  conditions (ii) and (iii) of Lemma \ref{mainlemma} follow. Thus combining all the steps we have Theorem \ref{thm:lsd}(i).
 The proof of Theorem \ref{thm:lsd}(ii) is provided in Section \ref{sec:thelsd}. 

By adapting the techniques of \cite{bryc2006spectral,bose2008another}  we can conclude that the limit has unbounded support and is not unimodal for all $\theta$. We omit the details of these arguments. 

It seems hard to deduce any further properties of the LSD for general $\theta \in (0,\infty)$. Nevertheless, with a combinatorial argument, we are able to identify $F_\theta$ when $\theta$ is an integer.  
It is known that $F_1$ is the symmetrized  standard Rayleigh distribution  \cite[see][]{bose2002limiting}. Let  $\mathcal{R}\sim F_1$. 
The following  theorem describes $F_\theta$ for integer $\theta$. 
Its proof is provided in Section \ref{sec:wordlimits}.

\begin{theorem}\label{thm:theta_integer}
Let $\theta$ be a positive integer. Suppose $X_\theta \sim F_\theta$. 
Let  $ B \sim \text{Ber}\Big(\frac{1}{\theta}\Big)$ 
and $B$ and $\mathcal R$ be  independent. 
Then \ $X_\theta \stackrel{d}{=} B \sqrt{\theta} \mathcal{R}$.
\end{theorem}
\begin{remark}\label{rem:prop_zero}
It follows from Theorem~\ref{thm:theta_integer} that the proportion of zero eigenvalues in $A_n^{\theta}$, $\theta \in \N$ converges almost surely to $1-1/\theta$ as $n \rightarrow \infty$. The same conclusion continues to hold for non-integer values of $\theta$ as well. Indeed, note that the $\lfloor n/\theta \rfloor \times  \lfloor n/\theta \rfloor$ principal submatrix of $A_n^\theta$, denoted by $A_n^{\theta,1}$, is a $\lfloor n/\theta \rfloor \times  \lfloor n/\theta \rfloor$ Reverse Circulant matrix, and that $\text{rank}(A_{n}^{\theta,1})/\lfloor n/\theta \rfloor \xrightarrow{a.s.} 1-P(\mathcal{R}=0)=1$. Now since $\text{rank}(A_n^\theta)=\text{rank}(A_n^{\theta,1})$, one deduces that the proportion of zero eigenvalues in $A_n^{\theta}$ equals $1-\text{rank}(A_n^{\theta,1})/n\stackrel{a.s.}\sim 1-\lfloor n/\theta \rfloor/n\sim 1-1/\theta.$
\end{remark}
\begin{remark}
It follows from Remark~\ref{rem:prop_zero} that $F_{\theta}$ can be represented as
\[
F_{\theta}(\cdot)=\left(1-\frac{1}{\theta}\right)\Delta_0 (\cdot)+\frac{1}{\theta}G_{\theta}(\cdot),
\]
for some distribution function $G_{\theta}(\cdot)$, where $\Delta_0(\cdot)$ is the distribution function corresponding to $\delta_0$. Theorem~\ref{thm:theta_integer} implies that for integer values of $\theta$, $G_{\theta}(\cdot)$ equals $F_1(\frac{1}{\sqrt{\theta}}\cdot)$, the distribution function of $\sqrt{\theta}\mathcal{R}$. From Example~\ref{exm:comp} it follows that for non-integer values of $\theta$, $G_{\theta}(\cdot)\ne F_1(\frac{1}{\sqrt{\theta}}\cdot)$.
\end{remark}
See Figure~\ref{fig:gen_hankel} for some simulations with different values of $\theta$. See Figure~\ref{fig:comparison} for a comparison of $G_{\theta}(\cdot)$ with $F_1(\frac{1}{\sqrt{\theta}}\cdot)$, for non-integer values of $\theta$.
\begin{figure}[!t]
\centering
$\begin{array}{cc}
\includegraphics[width=0.4\textwidth, height=5cm]{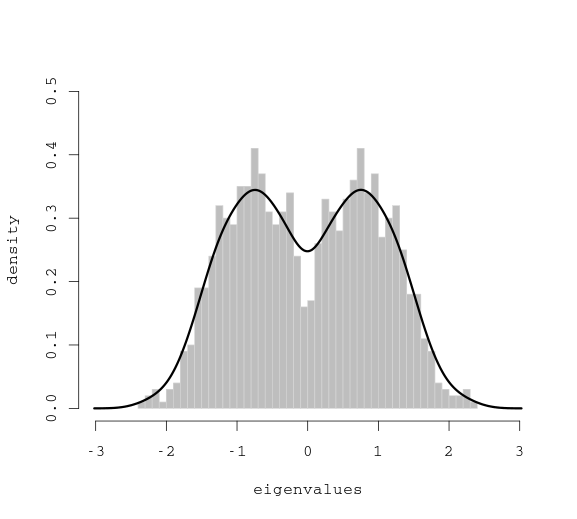}&
\includegraphics[width=0.4\textwidth, height=5cm]{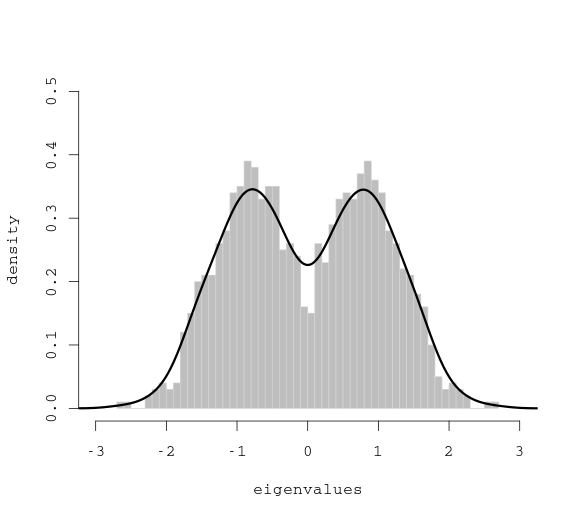}\\
\includegraphics[width=0.4\textwidth, height=5cm]{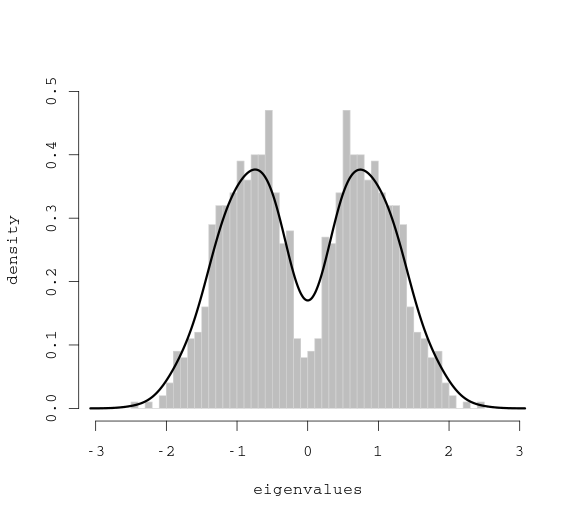}&
\includegraphics[width=0.4\textwidth, height=5cm]{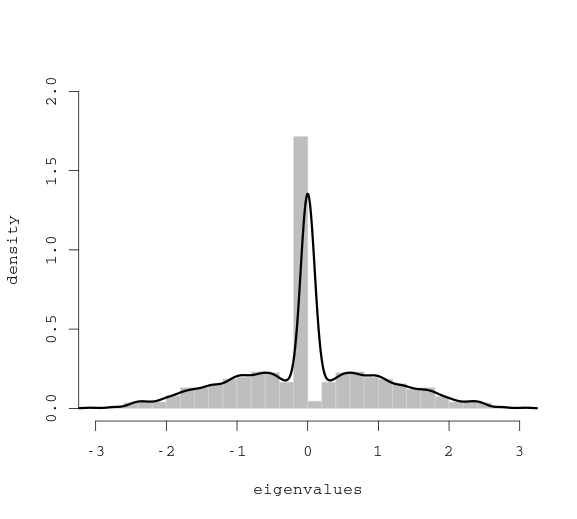}\\
\includegraphics[width=0.4\textwidth, height=5cm]{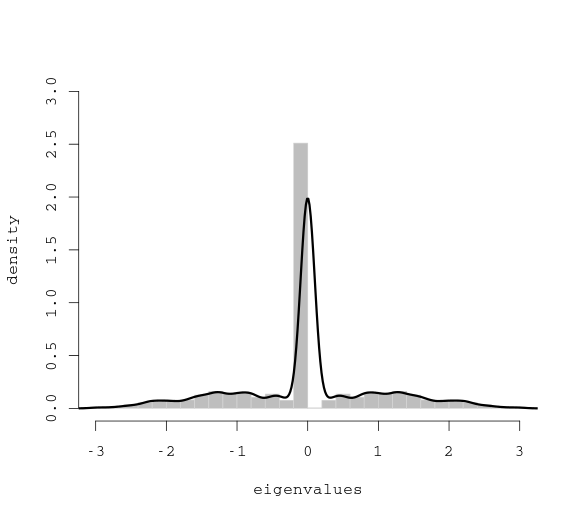}&
\includegraphics[width=0.4\textwidth, height=5cm]{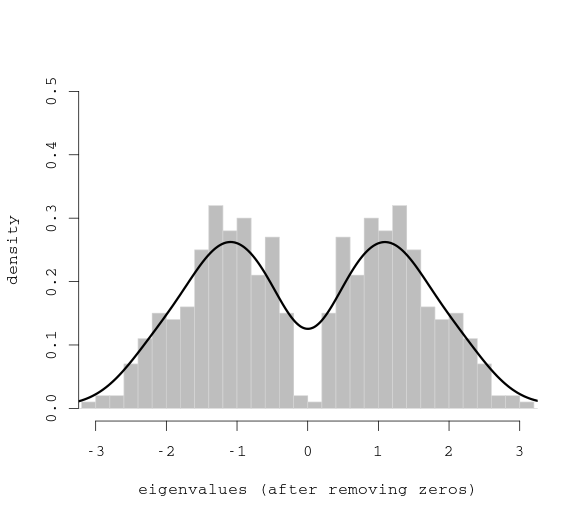}\\

\end{array}$
\caption{\footnotesize{\textit{Top left:} $\theta=0.5$, i.e. Hankel. \textit{Top right:} $\theta=0.75$. \textit{Middle left:} $\theta=1$, i.e. Reverse Circulant. \textit{Middle right:} $\theta=1.5$. \textit{Bottom left:} $\theta=2$. \textit{Bottom right:} $\theta=2$ with the zero eigenvalues removed. $\mathcal{N}(0,1)$ entries were used in each model with $n=1000$.}}
\label{fig:gen_hankel}
\end{figure}

\begin{figure}[!t]
\centering
$\begin{array}{cc}
\includegraphics[width=0.4\textwidth, height=5cm]{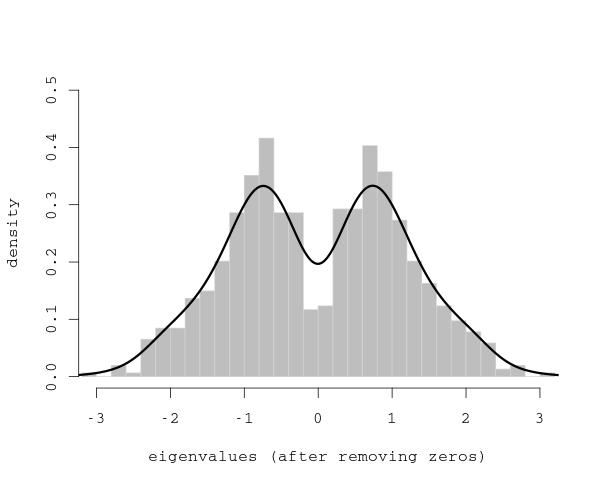}&
\includegraphics[width=0.4\textwidth, height=5cm]{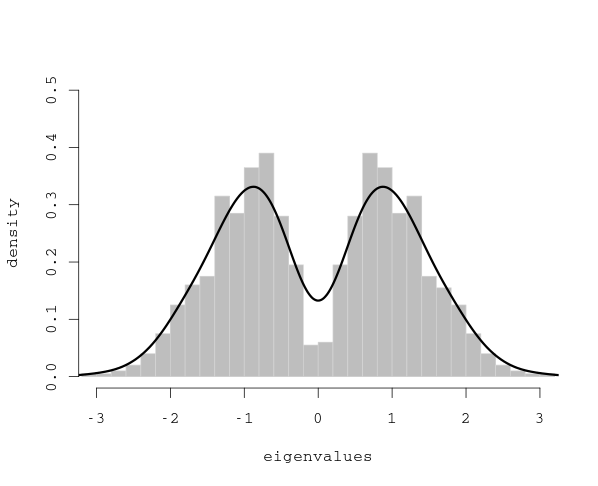}\\
\includegraphics[width=0.4\textwidth, height=5cm]{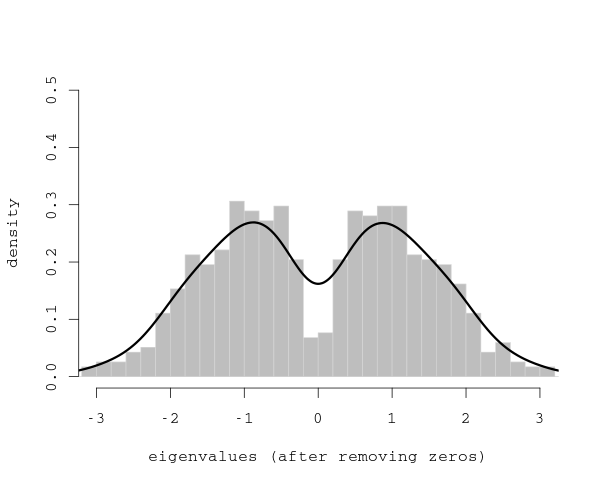}&
\includegraphics[width=0.4\textwidth, height=5cm]{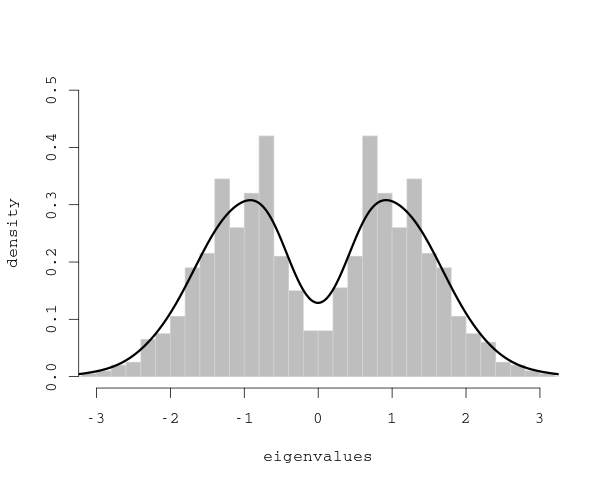}\\
\end{array}$
\caption{\footnotesize{\textit{Top left:} $\theta=1.3$, with zero eigenvalues removed. \textit{Top right:} ESD of $\sqrt{\theta}A_n^{1}$ (whose LSD is $F_1(\frac{1}{\sqrt{\theta}}\cdot)$), with $\theta=1.3$. \textit{Bottom left:} $\theta=1.7$, with zero eigenvalues removed. \textit{Bottom right:} ESD of $\sqrt{\theta}A_n^{1}$, with $\theta=1.7$. $\mathcal{N}(0,1)$ entries were used in each model with $n=1000$.}}
\label{fig:comparison}
\end{figure}

\section{The LSD}\label{sec:thelsd}
\begin{proof}[Proof of Theorem \ref{thm:repr}]
Consider a $\theta$-pair-matched circuit $\pi\in \Pi^*_\theta(w)$. 
If $w[i]=w[j]$ then 
\begin{align}\label{rel1}
& L_{\theta}(\pi(i-1),\pi(i))=L_{\theta}(\pi(j-1),\pi(j)) \nonumber \\ 
\Leftrightarrow & \,  \pi(i-1)+\pi(i) \equiv \pi(j-1)+\pi(j) (\text{mod}\lfloor n/\theta \rfloor).
\end{align}
If we define $\nu_i :=\pi(i)/\lfloor n/\theta \rfloor$ and $t_i := \nu_{i-1}+\nu_i$, then (\ref{rel1}) is equivalent to
\begin{equation}
t_i-t_j=r,
\end{equation}
where $r\in \{0, \pm 1, \pm 2, \ldots \pm K_{\theta,n}\}$, and $K_{\theta,n}\geqslant 0$ is the greatest integer satisfying the restriction
\begin{equation}
K_{\theta,n}\lfloor n/\theta \rfloor \leqslant 2n-2,
\end{equation}
which reflects the fact that $|\pi(i-1)+\pi(i)-\pi(j-1)-\pi(j)|\leqslant 2n-2$. The constraint on $K_{\theta,n}$ can be rewritten as
\begin{equation}\label{eq:bounds_on_K}
K_{\theta,n} \leqslant 2\frac{n}{\lfloor n/\theta \rfloor} - \frac{2}{\lfloor n/\theta \rfloor}.
\end{equation}
\textbf{Case I: $2\theta$ is not an integer.}  Note that $\frac{n}{\lfloor n/\theta \rfloor} \rightarrow \theta$ as $n\rightarrow \infty.$ In the present case $\{2\theta\}$, the fractional part of $2\theta$, is positive. We claim that for all sufficiently large $n$ we have
\begin{equation}
K_{\theta,n}=\lfloor 2\theta \rfloor.
\end{equation}
Indeed, on writing $n/\theta=\lfloor n/\theta \rfloor+\{n/\theta\}$, we have from inequality (\ref{eq:bounds_on_K})
\begin{align*}
K_{\theta,n} \leqslant 2\frac{n}{\lfloor n/\theta \rfloor} - \frac{2}{\lfloor n/\theta \rfloor}& =2\theta+\frac{2\theta\{n/\theta\}}{\lfloor n/\theta \rfloor}-\frac{2}{\lfloor n/\theta \rfloor}\\
&=\lfloor2\theta \rfloor + \{2\theta\}+\frac{2\theta\{n/\theta\}}{\lfloor n/\theta \rfloor}-\frac{2}{\lfloor n/\theta \rfloor}.
\end{align*}
By our assumption, $\{2\theta\}>0$ and hence for all large $n$,  $\lfloor n/\theta \rfloor >0$ and the part $2\theta\{n/\theta\}-2$ remains bounded. As a consequence, for all large enough $n$ we have
\begin{equation}\label{eq:nonneg}
0<\{2\theta\}+\frac{2\theta\{n/\theta\}}{\lfloor n/\theta \rfloor}-\frac{2}{\lfloor n/\theta \rfloor}<1.
\end{equation}
Therefore, for all large enough $n$,
\[
K_{\theta,n}=\lfloor 2\theta \rfloor.
\]

Let $(i_s, j_s)$, $1\leqslant s \leqslant k$ denote  all the $k$ matches in $w$.
Let $\gamma_s:=t_{i_s}-t_{j_s} \in \mathcal{A}_{\theta}$ and $\gamma:=(\gamma_s,\, 1\leqslant s \leqslant k)$. Arguing as in \cite[Theorem 8]{bose2008another} we obtain 
\[
\# \Pi^*_\theta(w)= \sum_{\gamma\in \mathcal{A}_{\theta}^k} \# \{ (\nu_0,\dots \nu_{2k})\mid \nu_0=\nu_{2k}, \, \nu_i \in U_{\theta,n}, \text{ and } t_{i_s}-t_{j_s} = \gamma_s\},
\]
where $U_{\theta,n}:=\{ 1/\lfloor n/\theta \rfloor, \dots , n/\lfloor n/\theta \rfloor \}$.
Furthermore, for $i\notin S(w)$, one can check that
\[
\nu_i=L_i^H(\nu_{S(w)})+a_i^{(\gamma)},
\]
where $L_i^H(\nu_{S(w)})$ denotes a certain linear combination in the variables $\{\nu_i, \, i\in S(w)\}$\footnote{More precisely $L_i^H(\cdot)$ is the linear map appearing in \cite[proof of Theorem 6]{bose2008another}.} and $a_i^{(\gamma,w)}:=a_i^{(\gamma)}$ is some integer dependent on $\gamma$, and possibly on the word $w$. Note that $U_{\theta, n}$ is a discrete approximation to the interval $(0,\theta)$. So, $\frac{1}{\lfloor n/\theta \rfloor^{1+k}}\# \Pi^*_\theta(w)$ is a Riemann approximation to the integral
\[
{\sum_{\gamma\in \mathcal{A}_{\theta}^k}\int_{0}^{\theta}\cdots \int_{0}^{\theta} \I\, (L_i^H(\nu_{S(w)})+a_i^{(\gamma)} \in (0,\theta); { \, i \notin S(w) \cup \{2k\}})\I(\nu_0=L_{2k}^H(\nu_S(w))+a_{2k}^{(\gamma)})\, \text{d}\nu_{S(w)}.}
\]
It follows, since $\frac{\lfloor n/\theta \rfloor}{n}\rightarrow \theta$, that $p_\theta(w)=\lim_n \frac{1}{n^{1+k}}\#\Pi^*_\theta(w)$ exists and is given by 
\[
\frac{1}{\theta^{k+1}}\sum_{\gamma\in \mathcal{A}_{\theta}^k}\int_{0}^{\theta}\cdots \int_{0}^{\theta} \I\, (L_i^H(\nu_{S(w)})+a_i^{(\gamma)} \in (0,\theta); { i \notin S(w) \cup \{2k\}})\I(\nu_0=L_{2k}^H(\nu_{S(w)})+a_{2k}^{(\gamma)})\, \text{d}\nu_{S(w)}.
\]
\textbf{Case II: $2\theta$ is an integer.}  Since $\{2\theta\}=0$, we do not have (\ref{eq:nonneg}) here. Therefore it is not immediate that $K_{\theta,n}=\lfloor 2 \theta \rfloor = 2 \theta$ for all large values of $n$. Indeed, we will see that $K_{\theta,n}$ can assume different values (namely $\lfloor 2 \theta\rfloor$ and $\lfloor 2 \theta \rfloor -1$) along different subsequences. We circumvent this problem by showing that the value of the integral in (\ref{gen_hankel_word_limit}) is $0$, if $\gamma_i= \pm \lfloor 2 \theta \rfloor$ for some $i \in S(w)$. We now proceed to prove these claims.
 
First consider an integer $n$ of the form $Mn_1+n_2$, where $\theta=M/2$ for some positive integer $M$, and $n_1, n_2$ are positive integers with $n_2<M/2$. Then (\ref{eq:bounds_on_K}) implies that 
 \[
 K_{\theta,n}\leqslant M+\frac{n_2-1}{n_1}.
 \]
 Since for all large $n$ of this form one has
 \[
 0\leqslant\frac{n_2-1}{n_1}<1,
 \]
 we conclude that for all large $n$ of this form we have
 \[
  K_{\theta,n}=M=\lfloor 2\theta \rfloor.
 \]
 So, the arguments in Case I apply verbatim for this subsequence 
 \[
 \{n_\ell^{(1)}\}=\{Mn_1+n_2\}_{\begin{subarray}{1} n_1>0\\0<n_2<M/2\end{subarray}},
 \] 
 and therefore the corresponding subsequential word limit exists and the limit is given by the formula (\ref{gen_hankel_word_limit}). Similarly considering the subsequence
 $$\{n_\ell^{(2)}\}=\{Mn_1+n_2\}_{\begin{subarray}{1} n_1>0\\M/2<n_2<M\end{subarray}},$$
it can be checked that $K_{\theta,n}= \lfloor 2 \theta \rfloor$, for all large $n$. Therefore the same conclusion holds. Now  for the subsequence 
\[
 \{n_\ell^{(3)}\}=\{Mn_1+n_2\}_{\begin{subarray}{1} n_1>0\\n_2 \in \{0,M/2\}\end{subarray}},
 \]
 it is not hard to verify that $K_{\theta,n}= \lfloor 2 \theta \rfloor -1$, for all large $n$. Thus repeating
 the same ideas as before, along the subsequence $\{n_\ell^{(3)}\}$, the limit of $ \frac{1}{n^{1+k}} \# \Pi_\theta^*(w)$ exists, and is given by
 \begin{align}
  p_\theta^{(3)}(w):= \frac{1}{\theta^{k+1}}\sum_{\gamma\in \widetilde{\mathcal{A}}_{\theta}^k}\int_{0}^{\theta}\cdots \int_{0}^{\theta} \I\, (L_i^H(\nu_{S(w)})& \, +a_i^{(\gamma)} \in (0,\theta); { i \notin S(w) \cup \{2k\}})\label{eq:p_theta_3}\\
& \,   \times \I(\nu_0=L_{2k}^H(\nu_{S(w)})+a_{2k}^{(\gamma)})\, \text{d}\nu_{S(w)}, \notag
 \end{align}
  where $\widetilde{\mathcal{A}}_\theta= \{0, \pm 1, \ldots, \pm (\lfloor 2 \theta \rfloor -1) \}$. We now proceed to prove that the integral in (\ref{gen_hankel_word_limit}) is $0$, if $\gamma_i = \pm \lfloor 2 \theta \rfloor$, for some $i \in S(w)$. First fix a pair-matched word $w$ and the subsequence $\{n_\ell^{(1)}\}$.  For all $i \in S(w)$, define
  \[
  \Pi^{**, \pm}_{\theta,i}(w):= \{\pi \in \Pi^*_\theta(w): \, \pi(i-1)+ \pi(i) - \pi(j-1)- \pi(j) = \pm \lfloor 2\theta \rfloor \lfloor n/ \theta \rfloor\},
  \]
  where $w[i]=w[j]$. Denoting 
  $$\Delta_{i-1} := n - \pi(i-1), \, \Delta_{i}:= n-\pi(i), \, \Delta_{j-1}:= \pi(j-1)-1, \text{ and } \Delta_j:= \pi(j) -1,$$
 we note that
 \begin{align}
& \, \,  \pi(i-1)+ \pi(i) - \pi(j-1)- \pi(j) =  \lfloor 2\theta \rfloor \lfloor n/ \theta \rfloor \notag\\
\Leftrightarrow & \, \, \Delta_{i-1}+\Delta_i + \Delta_{j-1}+ \Delta_j = 2n_2 -2,
 \end{align}
 for the subsequence $\{n_\ell^{(1)}\}$. Since all the $\Delta$'s are nonnegative, we deduce that the quadruple $(\pi(i-1), \pi(i), \pi(j-1), \pi(j))$ can be chosen only in $O(1)$ ways, which depend only on $M$. This in particular fixes the allowable choices of the generating vertex $\pi(i)$ to be $O(1)$, and thus by standard argument $\#\Pi_{\theta,i}^{**,+}(w)=O(n^k)$. By a similar argument one also has that $\#\Pi_{\theta,i}^{**,-}(w)=O(n^k)$. On the other hand, note that ${(\theta/n)^{1+k}} \# \Pi_{\theta, i}^{**, \pm}(w)$ is a discrete approximation of the integral
 $$\sum_{\begin{subarray}{1}\gamma \in \mathcal{A}_\theta^k\\\gamma_i = \pm \lfloor 2 \theta \rfloor\end{subarray}}\int_{0}^{\theta}\cdots \int_{0}^{\theta} \I\, (L_i^H(\nu_{S(w)})  +a_i^{(\gamma)} \in (0,\theta); { i \notin S(w) \cup \{2k\}})   \times \I(\nu_0=L_{2k}^H(\nu_{S(w)})+a_{2k}^{(\gamma)})\, \text{d}\nu_{S(w)}.
 $$
 Since $\#\Pi_{\theta,i}^{**, \pm}(w)= O(n^k)$,  it immediately follows that  the integral above is  zero. Since this holds for all $i \in S(w)$,  the limit for the subsequence $\{n_\ell^{(1)}\}$ must be as given in (\ref{eq:p_theta_3}). Similar conclusions can also be made for the subsequence $\{n_\ell^{(2)}\}$. We omit the details. Therefore we have shown that the limit along any subsequence is same, and is given by (\ref{eq:p_theta_3}). Therefore the proof is complete.
\end{proof}
\begin{remark}
For $\theta \leqslant 1/2$, since $\gamma_s$ can only take the value $0$, the integers $a_i^{(\gamma)}$ are all $0$. Therefore (\ref{gen_hankel_word_limit}) reduces to the  known formula for Hankel word limits upon a change of variable $\nu_j \rightarrow \nu_j/\theta,\, j\in S(w)$.
\end{remark}

\begin{proof}[Proof of Theorem \ref{thm:lsd}(ii)]
We will show that for all $h \ge 1$,  $\beta_h^{(\theta_m)} \rightarrow \beta_h^{(\theta)}$, whenever $\theta_m \rightarrow \theta \in (0,\infty)$. This will complete the proof since $\{\beta^{(\theta)}_h\}$ satisfies Carleman's condition. Noting that the functions $\{L_i^H(\cdot)\}_{i \notin S(w)}$ are all non-trivial linear functions, the conclusion is obvious when $2 \theta$ is not an integer. By a similar argument, when $2 \theta$ is an integer, one can immediately deduce the right continuity of $\beta_h^{(\theta)}$. For the left continuity, we begin by noting that $\lfloor 2 \theta_m\rfloor = \lfloor 2 \theta \rfloor -1$, for large $m$, and therefore
\begin{align*}
p_{\theta_m}(w)=  \frac{1}{\theta_m^{k+1}}\sum_{\gamma\in \widetilde{\mathcal{A}}_{\theta}^k}\int_{0}^{\theta_m}\cdots \int_{0}^{\theta_m} \I\, (L_i^H(\nu_{S(w)})& \, +a_i^{(\gamma)} \in (0,\theta_m); { i \notin S(w) \cup \{2k\}})\\
& \,   \times \I(\nu_0=L_{2k}^H(\nu_{S(w)})+a_{2k}^{(\gamma)})\, \text{d}\nu_{S(w)}.
\end{align*}
As noted in the proof of Theorem \ref{thm:repr}, for $2 \theta$ integer, the quantity $p_\theta(w)$ is also given by (\ref{eq:p_theta_3}), the conclusion follows by a similar argument as before.
\end{proof}

{\begin{remark}\label{genlink}
One might consider the following link functions
\begin{equation}
L_n(i,j)=i+j \text{ (mod }a_n),
\end{equation}
where $\{a_n\}$ is a sequence of positive integers. To prevent degeneracy we must have $a_n\rightarrow \infty$. Note that if $n/a_n\rightarrow 0$, we are back to the Hankel case. On the other hand if $n/a_n\rightarrow \theta$ for some positive real number $\theta$, such a link function will generalize $L_{\theta}$ (note that in this case we can write $a_n=\lfloor n/\theta \rfloor + o(n)$; thus, essentially such a link function is equivalent to $L_{\theta}$).  All our arguments can be modified, and Theorem \ref{thm:repr} continues to hold for such link functions. Finally, if $n/a_n\rightarrow \infty$, then the word limits diverge to $\infty$, when $k\geq 3$. This seems to be an interesting situation. However the moment method fails to shed any light here. Note that a different scaling will not work because with the scaling we have used, one still has $\beta_2=1$.
\end{remark}}
The integral in (\ref{gen_hankel_word_limit}) is very difficult to evaluate for arbitrary $w$, specially when $k$ is large. Below and also in Section \ref{sec:wordlimits}, we identify some specific words for which the integral is easy to compute, and thereby yields some more information about the limit.

A pair-matched word is said to be \textit{symmetric} if each letter appears once in an even position and once in an odd position. Examples include $abcabc$, $aabbcddc$ etc. In \cite{bose2008another} it is shown  that for the Hankel and Reverse Circulant matrices $p(w)=0$ for non-symmetric words. An adaptation of the same argument shows that for all $\theta$, $p_\theta(w)=0$, for all non-symmetric words $w$. We omit the mundane details. This observation immediately yields the following corollary and  proposition.
\begin{cor}[Super-Hankel limit] For all $\theta$, 
\[
\beta_{2k}^{(\theta)}\geqslant \beta_{2k}^{(1/2)}= \beta_{2k}^{(Hankel)}.
\]
\end{cor}
\begin{proof}
In (\ref{gen_hankel_word_limit}) consider the term with $\gamma=\gamma_0:=(0,\ldots,0)$. Since we have $a_i^{(\gamma_0)}=0$ for each $i\notin S(w)\cup \{ 2k \}$, the term is easily seen to equal $\beta_{2k}^{(Hankel)}$ upon a change of variable $\nu_j \rightarrow \nu_j/\theta,\, j\in S(w)$.
\end{proof}
\begin{prop}[Sub-Reverse Circulant limit]
If $\theta\in (\frac{1}{2},1)$, then 
\[
\beta_{2k}^{(\theta)} \leqslant \beta_{2k}^{(1)}=\beta_{2k}^{(Rev.\,  Circ.)}= k!.
\]
\end{prop}
\begin{proof}
In this case $\lfloor2\theta \rfloor=1$ and $\mathcal{A}_{\theta}=\{0,\pm 1\}$. Consider a symmetric word $w$. Following the argument in the Reverse Circulant case (see \cite{bose2008another}),  upon choosing the generating vertices, there is \emph{at most} one choice for each non-generating vertex. To elaborate, consider a match $(i,j)$  and suppose that $\pi(i-1), \pi(i)$, and $\pi(j-1)$ have been chosen. Now from the relation 
\[
\pi(j)=\pi(i-1)+\pi(i)-\pi(j-1)+r\lfloor n/\theta \rfloor,
\]
where $r\in \{0,\pm 1 \}$, since $\lfloor n/\theta \rfloor \geqslant n$  and $1 \leqslant \pi(j)\leqslant n$, it follows that there is at most one choice for $\pi(j)$. Therefore $\lim\frac{1}{n^{1+k}}\Pi^*_\theta(w)\leqslant 1$ and the stated assertion follows, because for Reverse Circulant each symmetric word contributes $1$.
\end{proof}

\section{Word limits}\label{sec:wordlimits}
 We can, in principle, obtain all the moments from  (\ref{gen_hankel_word_limit}). However, evaluating the integral therein is not an easy task, even in the Hankel case, and no explicit recursions or formulas are available for the moments. In this section,  we provide a recursion for $p_\theta(w)$, when $w$ is a Catalan word, and we also determine  $F_\theta$ completely, when $\theta$ is an integer.

A pair matched word is called \textit{Catalan}, if sequentially deleting all double letters reduces the word to an empty word (e.g. $aabb$, $abbcca$ etc.), and we let $\mathcal{C}_{2k}$ to be the class of all Catalan words of length $2k$. These are in bijection with the so called non-crossing pair-partitions (c.f. \cite{bose2008another}, and \cite[Proposition 2.1.11]{anderson2010introduction}).
Their importance in random matrix theory stems from the fact that for the Wigner matrix, $p(w)$ is non-zero only if the word is Catalan and  in which case $p(w)=1$. For Hankel, Toeplitz, Symmetric Circulant and Reverse Circulant matrices $p(w)$ continues to equal $1$ for all Catalan words. However, this does not remain true for all values of $\theta$. Nevertheless, using a direct counting approach we establish the following useful formula for Catalan words.
\begin{theorem}\label{thm:gen_hankel_cat}
If $w \in \mathcal{C}_{2k}$ then $p_\theta(w)$ is given by
\begin{equation}
p_\theta(w)=\sum_{j=0}^{k-1}A_j^{w} \lfloor \theta \rfloor^j (\lfloor \theta \rfloor +1)^{k-1-j},
\end{equation}
where the $A_j^{w}$ are some non-negative constants (depending on $\theta$) such that $\sum_{j=0}^{k-1}A_j^{w}=1.$
\end{theorem}

\begin{proof}
\textbf{Case I: $\theta$ is not an integer.}
We first prove a ``finite'' version of the result. Suppose $w \in \mathcal{C}_{2k}$. Then we claim that for all large enough $n$ we have
\begin{equation}
\frac{1}{n^{1+k}}\#\Pi^*_\theta(w)=\sum_{j=0}^{k-1}A_{j,n}^{w} \lfloor \theta \rfloor^j (\lfloor \theta \rfloor +1)^{k-1-j},
\end{equation}
where the $A_{j,n}^{w}$ are non-negative constants (depending on $\theta$) adding up to $1$.
We shall use induction on $k$ to prove the claim. The stated assertion is trivially true when $k=1$ (the only word in $\mathcal{C}_{2}$ is $aa$ and $\frac{1}{n^{2}}\#\Pi^*_\theta(aa)=1$). So, suppose that the assertion is true for $(k-1)$, where $k\geqslant 2$. Let $w\in \mathcal{C}_{2k}$. As each Catalan word has a double letter, suppose that $w[i_0]=w[i_0+1]$, for some $i_0\in\{1,\ldots,2k\}$. This implies that
\begin{align}\label{eq:definition_of_r}
\nonumber & \pi(i_0-1)+\pi(i_0)\equiv \pi(i_0)+\pi(i_0+1) (\text{mod $\lfloor n/\theta \rfloor$})\\
\nonumber \Rightarrow & \pi(i_0-1)-\pi(i_0+1)\equiv 0 (\text{mod $\lfloor n/\theta \rfloor$})\\
\Rightarrow & \pi(i_0-1)-\pi(i_0+1)=r\lfloor n/\theta \rfloor,  
\end{align}
for some integer $r \in \{0,\pm1,\ldots,\pm C_{\theta,n}\}$, where $C_{\theta,n}\geqslant0$ is the maximum integer satisfying (\ref{eq:definition_of_r}) and not violating the automatic restriction $|\pi(i_0-1)-\pi(i_0+1)|\leqslant n-1$.  Hence
\begin{align}
\nonumber C_{\theta,n} & \leqslant \frac{n}{\lfloor n/\theta \rfloor}-\frac{1}{\lfloor n/\theta \rfloor}\\
\nonumber &=\theta+\frac{\theta\{n/\theta\}}{\lfloor n/\theta \rfloor}-\frac{1}{\lfloor n/\theta \rfloor}\\
&=\lfloor\theta\rfloor + \{\theta\}+\frac{\theta\{n/\theta\}}{\lfloor n/\theta \rfloor}-\frac{1}{\lfloor n/\theta \rfloor}.
\end{align}
By our assumption $\{\theta\}>0$. So, for all large enough $n$ we have
\[
0<\{\theta\}+\frac{\theta\{n/\theta\}}{\lfloor n/\theta \rfloor}-\frac{1}{\lfloor n/\theta \rfloor}<1.
\]
Therefore for all large enough $n$
\begin{equation}\label{eq:large_C}
C_{\theta,n}= \lfloor\theta \rfloor.
\end{equation}
In the rest of the proof, we assume that $n$ is large enough so that (\ref{eq:large_C}) holds. Now since
\[
\pi(i_0+1)=\pi(i_0-1)-r\lfloor n/\theta \rfloor,
\]
for a fixed $\pi(i_0-1)$, there are either $\lfloor\theta \rfloor$ or $\lfloor\theta \rfloor+1$ many choices for $\pi(i_0+1)$. To see this, first note that, for large values of $n$ (possibly depending on $\theta$), we can write $n=\lfloor\theta \rfloor\lfloor n/\theta \rfloor + \xi_n$, where $0 < \xi_n < \lfloor n/\theta \rfloor.$ Denoting by $E_i$ the set $\{i\lfloor n/\theta \rfloor+1,i\lfloor n/\theta \rfloor+2,\ldots,i\lfloor n/\theta \rfloor+\lfloor n/\theta \rfloor\}$, we note that the set $\{1,2,\ldots,n\}$ can be partitioned as
\begin{equation}
\{1,2,\ldots,n\}=\left(\bigcup_{i=0}^{\lfloor\theta \rfloor-1}E_i\right) \cup \{\lfloor\theta \rfloor\lfloor n/\theta \rfloor+1,\ldots,n\}.
\end{equation}
Define $r_{i_0}^\theta(\pi):=\pi(i_0-1)(\text{mod }\lfloor n/\theta \rfloor)$. If $1\leqslant r_{i_0}^\theta(\pi) \leqslant \xi_n$, then clearly $\pi(i_0-1)$ can be in any one of the $\lfloor\theta \rfloor+1$ partition-blocks. Then $r$ can take one of exactly $\lfloor\theta \rfloor+1$ possible values so as to yield a feasible value for $\pi(i_0+1)$. On the other hand, if $\xi_n+1\leqslant r_{i_0}^\theta(\pi) \leqslant \lfloor n/\theta \rfloor-1$, then, by a similar reasoning, there are exactly $\lfloor\theta \rfloor$ many valid choices for $r$.

Now, deleting the first double letter from left, we obtain another Catalan word $\hat{w}\in \mathcal{C}_{2k-2}$. We now note that a circuit  $\hat{\pi}\in \Pi^*_\theta(\hat{w})$ gives rise to circuits $\pi \in \Pi^*_\theta(w)$ in the following way:
\[
\pi(i):=\begin{cases}
\hat{\pi}(i), & \text{ if $0\leqslant i\leqslant i_0-1$},\\
f_1, & \text{ if $i=i_0$},\\
f_2, & \text{ if $i=i_0+1$}, and\\
\hat{\pi}(i-2), & \text{ if $i_0+2\leqslant i\leqslant 2k$},
\end{cases}
\]
where $f_1$ is chosen arbitrarily in $n$ ways and $f_2$ is chosen in either $\lfloor\theta \rfloor+1$ or $\lfloor\theta \rfloor$ ways, according as whether $r_{i_0}^\theta(\hat{\pi})\in \{1,\ldots,\xi_n\}$ or $r_{i_0}^\theta(\hat{\pi})\in \{\xi_n+1,\ldots,\lfloor n/\theta \rfloor-1\}$. Thus if $r_{i_0}^\theta(\hat{\pi})\in \{1,\ldots,\xi_n\}$, then 
$\hat{\pi}$ gives rise to exactly $n(\lfloor\theta \rfloor+1)$ circuits $\pi \in \Pi^*_\theta(w)$ and if $r_{i_0}^\theta(\hat{\pi})\in \{\xi_n+1,\ldots,\lfloor n/\theta \rfloor -1\}$, then the number of such circuits is $n\lfloor\theta \rfloor$. For such circuits $\pi$,  we have $r_{i_0}^\theta(\pi)=r_{i_0}^\theta(\hat{\pi})$. For notational simplicity, let $r_{i_0}^\theta$ denote the common value.  Further let $d_{n,i_0}^{w}$  be the number of circuits $\pi\in\Pi^*_\theta(w)$ for which $r_{i_0}^\theta\in \{1,\ldots,\xi_n\}$, and define
\[
\alpha_{n,i_0}^{w}:=\frac{d_{n,i_0}^{w}}{\# \Pi^*_\theta(w)}.
\]
From the above discussion it is clear that
\begin{align*}
\#\Pi^*_\theta(w) & =n (\lfloor\theta \rfloor+1) d_{n,i_0}^{\hat{w}}  + n \lfloor\theta \rfloor (\# \Pi^*_\theta(\hat{w})-d_{n,i_0}^{\hat{w}}) \\
& = n \# \Pi^*_\theta(\hat{w}) (\alpha_{n,i_0}^{\hat{w}}(\lfloor\theta \rfloor+1) + (1-\alpha_{n,i_0}^{\hat{w}})\lfloor\theta \rfloor).
\end{align*}
Therefore
\begin{equation}\label{eq:recurrence_gen_hank_cat}
\frac{1}{n^{1+k}}\#\Pi^*_\theta(w) =\frac{1}{n^k}\# \Pi^*_\theta(\hat{w}) (\alpha_{n,i_0}^{\hat{w}}(\lfloor\theta \rfloor+1) + (1-\alpha_{n,i_0}^{\hat{w}})\lfloor\theta \rfloor).
\end{equation}
By the induction hypothesis,
\[
\frac{1}{n^{k}}\#\Pi^*_\theta(\hat{w})=\sum_{j=0}^{k-2}A_{j,n}^{\hat{w}} \lfloor\theta \rfloor^{j} (\lfloor\theta \rfloor+1)^{k-2-j}.
\]
Using this in (\ref{eq:recurrence_gen_hank_cat}) we obtain
\begin{align*}
\frac{1}{n^{1+k}}\#\Pi^*_\theta(w) & =\left(\sum_{j=0}^{k-2}A_{j,n}^{\hat{w}} \lfloor\theta \rfloor^{j} (\lfloor\theta \rfloor+1)^{k-2-j}\right)(\alpha_{n,i_0}^{\hat{w}}(\lfloor\theta \rfloor+1) + (1-\alpha_{n,i_0}^{\hat{w}})\lfloor\theta \rfloor)\\
& = \sum_{j=0}^{k-1}A_{j,n}^{w} \lfloor\theta \rfloor^{j} (\lfloor\theta \rfloor+1)^{k-1-j},
\end{align*}
where
\begin{equation}\label{eq:finite_rec}
A_{j,n}^{w}:=\begin{cases}
(1-\alpha_{n,i_0}^{\hat{w}})A_{j-1,n}^{\hat{w}} + \alpha_{n,i_0}^{\hat{w}}A_{j,n}^{\hat{w}}, & \text{for $j\in \{1,\ldots,k-2\}$},\\
\alpha_{n,i_0}^{\hat{w}}A_{j,n}^{\hat{w}}, & \text{for $j=0$, and}\\
(1-\alpha_{n,i_0}^{\hat{w}})A_{j-1,n}^{\hat{w}}, & \text{for $j=k-1$.}
\end{cases}
\end{equation}
Note that
\begin{align*}
\sum_{j=0}^{k-1}A_{j,n}^{w} &= \alpha_{n,i_0}^{\hat{w}}A_{j,n}^{\hat{w}} + \sum_{j=1}^{k-2}\left((1-\alpha_{n,i_0}^{\hat{w}})A_{j-1,n}^{\hat{w}} + \alpha_{n,i_0}^{\hat{w}}A_{j,n}^{\hat{w}}\right) + (1-\alpha_{n,i_0}^{\hat{w}})A_{j-1,n}^{\hat{w}}\\
&=  (1-\alpha_{n,i_0}^{\hat{w}})\sum_{j=0}^{k-2}A_{j,n}^{\hat{w}}+\alpha_{n,i_0}^{\hat{w}}\sum_{j=0}^{k-2}A_{j,n}^{\hat{w}}\\
&= \sum_{j=0}^{k-2}A_{j,n}^{\hat{w}}\\
&= 1, \, \text{(by induction hypothesis).}
\end{align*}
This completes the induction and the claim follows.

Now suppose that we could prove that $\alpha_{n,i_0}^{w}$ tends to a limit, say $\alpha_{i_0}^{w}$, for each Catalan word $w$, and for each $i_0\in \{1,\ldots,2k\}$, then using an induction argument with the recurrence (\ref{eq:finite_rec}) it would follow that $A_{j,n}^{w}$ tends to a limit, say $A_j^{w}$, for each Catalan word $w$ and for each $j\in\{0,\ldots,k-1\}$ (note again that the base case $k=1$ is trivial to verify), and we would have the following system of recurrences for $A_j^{w}$:
\begin{equation}
A_j^{w}=\begin{cases}
(1-\alpha_{i_0}^{\hat{w}})A_{j-1}^{\hat{w}} + \alpha_{i_0}^{\hat{w}}A_j^{\hat{w}}, & \text{for $j\in \{1,\ldots,k-2\}$},\\
\alpha_{i_0}^{\hat{w}}A_j^{\hat{w}}, & \text{for $j=0$, and}\\
(1-\alpha_{i_0}^{\hat{w}})A_{j-1}^{\hat{w}}, & \text{for $j=k-1$.}
\end{cases}
\end{equation}
In that case $p_\theta(w)$, for each Catalan word $w$, is given by
\begin{equation}
p_\theta(w)=\sum_{j=0}^{k-1}A_j^{w} \lfloor\theta \rfloor^j (\lfloor\theta \rfloor +1)^{k-1-j}.
\end{equation}
The constants $A_j^{w}$ clearly satisfy
\[
\sum_{j=0}^{k-1}A_j^{w}=1,
\]
because their finite $n$ counterparts do so.

It thus remains to prove that $\alpha_{n,i_0}^{w}$ tends to a limit for each Catalan word $w$, and for each $i_0\in \{1,\ldots,2k\}$. Let $\pi \in \Pi^*_\theta(w)$. Suppose that there are $m_{i_0}^{w}$ many vertices in $\pi$ that are dependent on $\pi(i_0-1)$ (i.e. choice of $\pi(i_0-1)$ constraints the choice of these vertices). A moment's reflection reveals that for a fixed $\pi(i_0-1)$ these vertices together can be chosen in either $(\lfloor\theta \rfloor+1)^{m_{i_0}^{w}}$ or $\lfloor\theta \rfloor^{m_{i_0}^{w}}$ ways, according as whether $r_{i_0}^\theta\in \{1,\ldots,\xi_n\}$ or $r_{i_0}^\theta\in \{\xi_n+1,\ldots,\lfloor n/\theta \rfloor-1\}$ (if $\pi(0)$ and $\pi(2k)$ both depend on $\pi(i_0-1)$, we count only one of them, in order to honor the automatic constraint $\pi(0)=\pi(2k)$). The other vertices can be chosen independently in, say, total $N$ ways.
Then,
\[
d_{n,i_0}^{w}=N(\lfloor\theta \rfloor+1)^{m_{i_0}^{w}}\xi_n (\lfloor\theta \rfloor+1),
\]
and
\[
\#\Pi^*_\theta(w)=N(\lfloor\theta \rfloor+1)^{m_{i_0}^{w}}\xi_n (\lfloor\theta \rfloor+1)+N \lfloor\theta \rfloor^{m_{i_0}^{w}} (\lfloor n/\theta \rfloor-\xi_n )\lfloor\theta \rfloor. 
\]
So,
\[
\alpha_{n,i_0}^{w}=\frac{(\lfloor\theta \rfloor+1)^{m_{i_0}^{w}+1}\xi_n}{(\lfloor\theta \rfloor+1)^{m_{i_0}^{w}+1}\xi_n +\lfloor\theta \rfloor^{m_{i_0}^{w}+1} (\lfloor n/\theta \rfloor-\xi_n )}.
\]
Now note that
\[
\frac{\xi_n}{n}=\frac{n-\lfloor\theta \rfloor\lfloor n/\theta \rfloor}{n}=1-\lfloor\theta \rfloor\frac{\lfloor n/\theta \rfloor}{n}\rightarrow 1-\frac{\lfloor\theta \rfloor}{\theta},
\]
and
\[
\frac{\lfloor n/\theta \rfloor - \xi_n}{n}=\frac{\lfloor n/\theta \rfloor}{n}-\frac{\xi_n}{n}\rightarrow \frac{1}{\theta}-(1-\frac{\lfloor\theta \rfloor}{\theta})=\frac{\lfloor\theta \rfloor+1}{\theta}-1.
\]
Therefore
\begin{equation}\label{eq:alpha}
\alpha_{n,i_0}^{w}\rightarrow \alpha_{i_0}^{w}:=\frac{(\lfloor\theta \rfloor+1)^{m_{i_0}^{w}+1}(1-\frac{\lfloor\theta \rfloor}{\theta})}{(\lfloor\theta \rfloor+1)^{m_{i_0}^{w}+1}(1-\frac{\lfloor\theta \rfloor}{\theta}) +\lfloor\theta \rfloor^{m_{i_0}^{w}+1} (\frac{\lfloor\theta \rfloor+1}{\theta}-1)}.
\end{equation}
\textbf{Case II: $\theta$ is an integer.} In this case $\lfloor\theta \rfloor=\theta$. Similar to the proof of Theorem \ref{thm:repr}, we consider several subsequences: First consider the subsequence $\{n_\ell^{(1)}\}$, where $n$ is a multiple of $\theta$, i.e. $n=\theta n_1$, for some $n_1$. Then $\{n/\theta\}=\{n_1\}=0$ and therefore
\[
C_{\theta,n}\leqslant \theta + \frac{1}{\lfloor n/\theta \rfloor}(\theta\{n/\theta\}-1) = \theta-\frac{1}{n_1},
\]
which implies that for all large enough $n$ (for the subsequence $\{n_\ell^{(1)}\}$) we have
\[
C_{\theta,n}=\theta-1.
\]
On the other hand, for the subsequence
$$\{n_\ell^{(2)}\}=\{\theta n_1+n_2\}_{\begin{subarray}{1} 0 < n_2 < \theta\\n_1 \in \mathbb{N}\end{subarray}},$$
we note that $\theta\{n/\theta\}=n_2$. As a consequence,
\[
C_{\theta,n}\leqslant \theta + \frac{1}{\lfloor n/\theta \rfloor}(\theta\{n/\theta\}-1) = \theta +\frac{n_2-1}{n_1}.
\]
So, in this case, for all large enough $n$, we have
\[
C_{\theta,n}=\theta.
\]
Repeating the arguments given in Case I, for the subsequence $\{n_\ell^{(1)}\}$, for $w \in \mathcal{C}_{2k}$, we obtain the limit to be
$$p_\theta(w)= \sum_{j=0}^{k-1}A_j^{w,1} (\theta-1)^j \theta^{k-1-j}.$$
On the other hand, for the subsequence $\{n_\ell^{(2)}\}$, the limit is given by
$$p_\theta(w)= \sum_{j=0}^{k-1}A_j^{w,2} \theta^j (\theta+1)^{k-1-j}.$$
Since the limit is same for all subsequence (as seen in the proof of Theorem \ref{thm:repr}), we must have
\begin{equation}\label{eq:theta_integral_cat}
p_\theta(w)=\sum_{j=0}^{k-1}A_j^{w,1} (\theta-1)^j \theta^{k-1-j}=\sum_{j=0}^{k-1}A_j^{w,2} \theta^j (\theta+1)^{k-1-j}.
\end{equation}
Defining $A_j^{w}=A_j^{w,2}$, we obtain the desired representation. This completes the proof of the theorem. 
\end{proof}

\noindent
Using the representation of $p_\theta(w)$ for $w \in \mathcal{C}_{2k}$, we deduce the following results:
\begin{cor}\label{cor1:cat}
If $\theta\leqslant 1$, then for each $w \in \mathcal{C}_{2k}$, $p_\theta(w)=1$.
\end{cor}
\begin{proof}
For $\theta<1$, $\lfloor\theta \rfloor=0$. Therefore, from the above theorem, for any Catalan word $w$ we have 
$p_\theta(w)=A_0^w.$ Using the recursion for $A_j^w$, we have $A_0^w=\alpha_{i_0}^{\hat{w}}A_0^{\hat{w}}$. Using $C_{\theta,n}=0$ for all large $n$, and (\ref{eq:alpha}), it is easy to check that $\alpha_{i_0}^{\hat{w}}=1$ for each $w$ and $i_0$. Therefore $A_0^{w}=A_0^{\hat{w}}$. Using this repeatedly we have $A_0^{w}=A_0^{aa}=1$.

If $\theta=1$ (it is already known that in the Reverse Circulant case the Catalan words contribute $1$; we still give a proof), from the inequality $C_{\theta,n}\leqslant \frac{n-1}{\lfloor n/\theta \rfloor}$ we have $C_{1,n}\leqslant 1 - \frac{1}{n}$, implying that $C_{1,n}=0$ for all $n$. So, the proof of Theorem \ref{thm:gen_hankel_cat} can be adapted to this case with $\lfloor\theta \rfloor$ replaced by $0$.
\end{proof}
\begin{cor}\label{cor2:cat}
If $\theta\in (1,\infty)$, then for each $w \in \mathcal{C}_{2k}$, we have $p_\theta(w)>1$. So, this is one example where the Catalan words contribute more than $1$ to the limiting moments.
\end{cor}
\begin{proof}
For $ \theta  \geqslant 2$, the proof follows straightaway from Theorem \ref{thm:gen_hankel_cat}. When  $\theta \in (1,2)$,  by (\ref{eq:alpha}), for every $i_0$, we have $\alpha_{i_0}^w >0$. Therefore we must have $A_{k-1}^w <1$, which in turn implies that $A_j^w>0$ for some $j \in \{0,1,\ldots,k-2\}$, and thus we have the desired result.
\end{proof}
\begin{cor}\label{cor:cat3}
If $\theta\in(1,\infty)$, then for each $w \in \mathcal{C}_{2k}$ we have
\[
p_\theta(w)\sim \theta^{k-1}, \text{ as $\theta\rightarrow \infty$}. 
\]
\end{cor}
\begin{proof}
It suffices to note that from Theorem \ref{thm:gen_hankel_cat} we have the estimate
\[
\lfloor\theta \rfloor^{k-1}\leqslant p_\theta(w) \leqslant (\lfloor\theta \rfloor + 1)^{k-1}.
\]
\end{proof}
\begin{cor}\label{cor4:cat}
If $\theta$ is a positive integer, we have $p_\theta(w)=\theta^{k-1}$, for all $w \in \mathcal{C}_{2k}$.
\end{cor}
\begin{proof}
We shall compute the constants $A_j^{w,1}$, using the recursions derived in the proof of Theorem \ref{thm:gen_hankel_cat}.
For all large $n$ in the subsequence $\{n_\ell^{(1)}\}$, we have $C_{\theta,n}=\theta-1$. Using this in place of $\lfloor \theta \rfloor$ in (\ref{eq:alpha}) we have the limiting $\alpha_{i_0}^{w,1}=1$ for all $i_0$ and $w \in \mathcal{C}_{2k}$. Therefore, using the recursion for the $A_j^{w,1}$, we obtain
\begin{equation}
A_j^{w,1}=\begin{cases}
A_j^{\hat{w},1}, & \text{for $j\in \{1,\ldots,k-2\}$},\\
A_j^{\hat{w},1}, & \text{for $j=0$, and}\\
0, & \text{for $j=k-1$,}
\end{cases}
\end{equation}
where $\hat{w}$ is the word obtained from $w$ by deleting the first double letter. Denoting $\hat{\hat{w}}$ to be the word obtained from $\hat{w}$ by removing the first double letter, and applying the same argument on $\hat{w}$ yields
\begin{equation}
A_j^{\hat{w},1}=\begin{cases}
A_j^{\hat{\hat{w}},1}, & \text{for $j\in \{1,\ldots,k-3\}$},\\
A_j^{\hat{\hat{w}},1}, & \text{for $j=0$, and}\\
0, & \text{for $j=k-2$.}
\end{cases}
\end{equation}
Continuing this procedure we finally obtain that $A_j^{w,1}=0$ for all $j\neq 0$ and $A_0^{w}=A_0^{aa}=1$. Substituting these values in (\ref{eq:theta_integral_cat}) we obtain the desired result.
\end{proof}
\begin{cor}\label{cor:cat4}
For any $\theta\in(1,\infty)$, we have the following weak bounds on the moments
\[
\frac{1}{k+1}\binom{2k}{k}\lfloor\theta \rfloor^{k-1} \leqslant \beta_{2k}^{(\theta)} \leqslant k! (\lfloor \theta \rfloor+1)^{k-1}.
\]
(These bounds are also true in the $\theta\leqslant 1$ regime, albeit the left inequality becomes trivial.) Thus, as $\theta \rightarrow \infty$, we have the weak asymptotic statement
\[
\beta_{2k}^{(\theta)}\asymp_k\theta^{k-1}.
\]
\end{cor}
\begin{proof}
The left inequality is obvious by Theorem \ref{thm:gen_hankel_cat}. The right side follows since there are $k!$ symmetric words and for each such word $p_\theta(w)$ is at most $(\lfloor\theta\rfloor+1)^k$. To see this note that if one chooses the generating vertices freely, then $\pi(2k)$ is fixed and each of the remaining $k-1$ vertices can be chosen in at most $\lfloor \theta \rfloor+1$ ways, so that $\#\Pi^*_\theta(w)\leqslant n^{1+k}(\lfloor \theta \rfloor+1)^{k-1}$. 
\end{proof}
{\begin{remark}
Theorem \ref{thm:gen_hankel_cat} and its corollaries continues to hold for the generalized link function $L_n$ of Remark \ref{genlink}, when $n/a_n \rightarrow \theta \in (0,\infty)$. 
\end{remark}
}
As a consequence of Theorem \ref{thm:gen_hankel_cat}, in Corollary \ref{cor4:cat} we already obtained the value of $p_\theta(w)$ for $w \in \mathcal{C}_{2k}$. Using a different argument we now find the value of $p_\theta(w)$ for all words in the integer $\theta$ case, which will establish then Theorem \ref{thm:theta_integer}.

\begin{proof}[Proof of Theorem \ref{thm:theta_integer}]
We already noted that $p_\theta(w)=0$ for any non-symmetric word. 
 Fix a symmetric pair-matched word $w$. We first show that  along the subsequence $\{n_\ell^{(1)}\}= \{\theta n_1\}_{n_1 \in \mathbb{N}}$, we have 
\begin{equation}
\#\Pi_\theta^*(w)= n^{1+k} \times \theta^{k-1},
\end{equation} thereby yielding $p_\theta(w)= \theta^{k-1}$, for such  a subsequence. 
To do this, we use the following combinatorial argument. 

We count the number of circuits $\pi \in \Pi^{*}_\theta(w)$, by identifying the number of possible choices for different vertices from left to the right. Obviously $\pi(0)$, and $\pi(1)$ can be chosen in $n$ valid ways. Similarly every generating vertex can be chosen in $n$ ways and all are valid choices. We further note that any non-generating vertex (excluding $\pi(2k)$), upon fixing the generating vertices to its left, can be chosen in exactly $\theta$ ways. Indeed, for a pair $(i_s,j_s)$, with $w[i_s]=w[j_s]$, we must have $\pi(j_s) \equiv r_{j_s}^\theta (\text{mod } n_1)$, where $n= \theta n_1$, and $r_{j_s}^\theta := \pi(i_s-1)+\pi(i_s) - \pi(j_s-1) \, \, (\text{mod } n_1)$. Thus there are exactly $\theta$ many choices for $\pi(j_s)$, namely $\{ r_{j_s}^\theta + \mu n_1, \, \mu =0,1,\ldots, \theta-1\}$ (when $r_{j_s}^\theta =0$, the range of $\mu$ changes to $\{1,\ldots,\theta\}$). Note that all of these choices are valid, and respect the pair-matched condition. After choosing the vertices $(\pi(0), \pi(1), \ldots, \pi(2k-1))$, it thus remains to argue that for any such choices, there exists one and only choice of $\pi(2k)$ obeying the circuit condition $\pi(0)=\pi(2k)$.  Note that given $(\pi(0), \pi(1), \ldots, \pi(2k-1))$ one can choose $\pi(2k)$ again in $\theta$ ways, obeying the pair-matched condition. Since $w$ is a pair-matched symmetric word, we have $\pi(0) - \pi(2k) = (s_1+s_3+\cdots+s_{2k-1}) - (s_2+s_4+\cdots+s_{2k}) = \lambda n_1$, for any such choice of $\pi(2k)$, where $\lambda$ is an integer and $s_i=\pi(i-1)+\pi(i)$. Noting that $|\pi(0)-\pi(2k)| \le n-1$, we further have that $\lambda \in \Lambda_\theta:=\{0, \pm1 , \ldots, \pm (\theta -1) \}$. Since $\#\Lambda_\theta = 2 \theta -1$, it is not immediate that one of the $\theta$ many choices of $\pi(2k)$ will automatically yield the circuit condition. However, fixing $ \lambda_1 n_1 +1 \le \pi(0) \le (\lambda_1+1) n_1$, for some $\lambda_1 \in \{0,1,\ldots, \theta-1\}$, we indeed have that $\lambda$ can take only $\theta$ many values, namely $\{\lambda_1-\theta+1,\ldots, \lambda_1 -1, \lambda_1\}$ (by noting that $1\leqslant \pi(2k)=\pi(0)+\lambda n_1\leqslant n=\theta n_1$). Thus for a fixed $\pi(0)$ one has exactly $\theta$ many possible choices for $\pi(2k)$ and these must match the $\theta$ many choices when one fills the circuit from left to right. Since $\lambda_1-\theta+1\leqslant 0 \leqslant \lambda_1$, there is one and only one choice of $\pi(2k)$ among the $\theta$ many possible choices that yields $\lambda=0$, i.e. the circuit condition $\pi(0)=\pi(2k)$. This completes the proof of the claim that $p_\theta(w)= \theta^{k-1}$, along the subsequence $\{n_\ell^{(1)}\}$. As noted in the proof of Theorem \ref{thm:repr}, the limit $p_\theta(w)$ is same along all other subsequences, for any pair-matched word $w$. Thus for all symmetric pair-matched word $w$ the limit is indeed $p_\theta(w)=\theta^{k-1}$, along all subsequences.  Hence $\beta_{2k}^{(\theta)}= k! \theta^{k-1}$, which clearly matches the moments of $X_\theta$. Therefore the proof is complete. 
\end{proof}
\begin{table}[htbp]
\centering
\begin{tabular}{|c|c|}
\hline
Word & $p_\theta(w)$\\
\hline
$abba$ & $\big(1-\frac{\lfloor\theta \rfloor}{\theta}\big)\big(\lfloor\theta \rfloor+1\big)^2+\big(\frac{\lfloor\theta \rfloor+1}{\theta}-1\big)\lfloor\theta \rfloor^2$\\
$aabb$ & $\big(1-\frac{\lfloor\theta \rfloor}{\theta}\big)\big(\lfloor\theta \rfloor+1\big)^2+\big(\frac{\lfloor\theta \rfloor+1}{\theta}-1\big)\lfloor\theta \rfloor^2$\\
\hline
\end{tabular}
\vskip5pt
\caption{Word limits when $w\in \mathcal{C}_{4}$.}
\label{table2}
\end{table}
\begin{exm}\label{exm:comp}
Table \ref{table2} gives the word limits for $k=2$.  For example,  
\[
\beta_4^{(\theta)}=p(abba)+p(aabb)=2\Big(1-\frac{\lfloor\theta \rfloor}{\theta}\Big)\big(\lfloor\theta \rfloor+1\big)^2+2\Big(\frac{\lfloor\theta \rfloor+1}{\theta}-1\Big)\lfloor\theta \rfloor^2.
\]
Note that this implies that for non-integer values of $\theta \geq 1$ the LSD is not given by $B\sqrt{\theta}\mathcal{R}$.
\end{exm}

\section{Acknowledgements}
The research of the second author was supported by J. C. Bose National Fellowship, Department of Science and Technology, Government of India.

\bibliographystyle{alpha}
\bibliography{generalized_hankel_final_revision}
\end{document}